\documentclass[leqno,12pt]{amsart} 
\usepackage{amssymb}
\theoremstyle{plain} 
\newtheorem{theorem}{\indent\sc Theorem}[section]
\newtheorem{lemma}[theorem]{\indent\sc Lemma}

\theoremstyle{definition} 

\newtheorem{remark}[theorem]{\indent\sc Remark}
\newtheorem{example}[theorem]{\indent\sc Example}

\def\C{{\mathbf{C}}}
\def\R{{\mathbf{R}}}
\def\Pi{{\mathbf{P}}}
\def\Si{{\mathbf{S}}}
\def\RC{{\overline{\mathbf{C}}}} 

\begin{document}

\title[Unicity theorem the Gauss maps]{A note on a unicity theorem for the Gauss maps of complete minimal surfaces in Euclidean four-space} 

\author[P.~H.~Ha]{Pham Hoang Ha} 

\author[Y.~Kawakami]{Yu Kawakami} 


\renewcommand{\thefootnote}{\fnsymbol{footnote}}
\footnote[0]{2010\textit{ Mathematics Subject Classification}.
Primary 53A10; Secondary 30D35, 53C42.}
%
%
\keywords{
minimal surface, Gauss map, unicity theorem.
}
\thanks{ 
The second author is supported by the Grant-in-Aid for Scientific Research (C), 
No. 15K04840, Japan Society for the Promotion of Science.}
\address{
Department of Mathematics,  \endgraf
Hanoi National University of Education, \endgraf
136, XuanThuy str., Hanoi, \endgraf
Vietnam
}
\email{ha.ph@hnue.edu.vn}

\address{
Faculty of Mathematics and Physics, \endgraf
Institute of Science and Engineering, \endgraf 
Kanazawa University, \endgraf
Kanazawa, 920-1192, Japan
}
\email{y-kwkami@se.kanazawa-u.ac.jp}


\maketitle

\begin{abstract}
The classical result of Nevanlinna states that two nonconstant meromorphic functions on the complex plane having the 
same images for five distinct values must be identically equal to each other. In this paper, we give a similar uniqueness theorem 
for the Gauss maps of complete minimal surfaces in Euclidean four-space. 
\end{abstract}

\section{Introduction} 
The Gauss map of a complete minimal surface in Euclidean space have some properties similar to 
the results in value distribution theory of a meromorphic function on the complex plane ${\C}$. One of the most notable results 
in this area is the Fujimoto theorem \cite[Theorem I]{Fu1988} which states that the Gauss map of a nonflat complete minimal surface 
in Euclidean $3$-space ${\R}^{3}$ can omit at most $4$ values. 
He also obtained the sharp estimate \cite[Theorem I\hspace{-.1em}I]{Fu1988} for the number of exceptional values of the Gauss map 
of a complete minimal surface in Euclidean $4$-space ${\R}^{4}$. Recently, 
the second author \cite{Ka2013} (for ${\R}^{3}$) and Aiyama, Akutagawa, Imagawa and the second author \cite{AAIK2016} (for ${\R}^{4}$) 
gave geometric interpretations of these Fujimoto results. 
Moreover Dethloff and the first author \cite{DH2014} proved ramification theorems for the Gauss map of complete minimal surfaces in 
${\R}^{3}$ and ${\R}^{4}$ on annular ends. Their results extended a result of Kao \cite{Ka1991}. 

Another famous result is the result on uniqueness and value sharing, which is called the unicity theorem. 
For meromorphic functions on $\C$, Nevanlinna \cite{Ne1926} proved 
that two meromorphic functions on $\C$ sharing $5$ distinct values must be identically equal to each other. 
Here we say that two meromorphic functions (or maps) $f$ and $\hat{f}$ share the value $\alpha$ (ignoring multiplicity) when 
$f^{-1}(\alpha) =\hat{f}^{-1}(\alpha)$. 
Fujimoto \cite{Fu1993-2} obtained the following analogy of this theorem for the Gauss maps of complete minimal surfaces in ${\R}^{3}$: 

\begin{theorem}\label{thm-Fujimoto-3}\cite[Theorem I]{Fu1993-2}
Let $X\colon \Sigma \to {\R}^{3}$ and $\widehat{X}\colon \widehat{\Sigma}\to {\R}^{3}$ be two nonflat minimal surfaces and 
$g\colon \Sigma \to \RC :=\C\cup \{\infty \}$, $\hat{g}\colon \widehat{\Sigma}\to \RC$ the Gauss maps of $X(\Sigma)$, 
$\widehat{X}(\widehat{\Sigma})$ respectively. Assume that there exists a conformal diffeomorphism $\Psi\colon \Sigma \to 
\widehat{\Sigma}$ and either $X(\Sigma)$ or $\widehat{X}(\widehat{\Sigma})$ is complete. 
If $g$ and $\hat{g}\circ \Psi$ share $7$ distinct values, then $g\equiv \hat{g}\circ \Psi$. 
\end{theorem}

We remark that the second author \cite{Ka2015} had unified explanation for the unicity theorem of the Gauss maps of several classes 
of surfaces in $3$-dimensional space forms including minimal surfaces in ${\R}^{3}$. 

The purpose of this paper is to give a similar uniqueness theorem for the Gauss maps of complete minimal surfaces in ${\R}^{4}$. 
The main theorem is stated as follows: 

\begin{theorem}\label{thm-cor}
Let $X\colon \Sigma \to {\R}^{4}$ and $\widehat{X}\colon \widehat{\Sigma} \to {\R}^{4}$ be two nonflat minimal surfaces, 
and $G=(g_{1}, g_{2})\colon \Sigma \to \RC\times \RC$, $\widehat{G}=(\hat{g}_{1}, \hat{g}_{2})\colon \widehat{\Sigma} 
\to \RC\times \RC$ the Gauss maps of $X(\Sigma)$, $\widehat{X}(\widehat{\Sigma})$ respectively. 
We assume that there exists a conformal diffeomorphism $\Psi \colon \Sigma \to \widehat{\Sigma}$ and 
either $X(\Sigma)$ or $\widehat{X}(\widehat{\Sigma})$ is complete. 
\begin{enumerate}
\item[(i)] Assume that $g_{1}, g_{2}, \hat{g}_{1}, \hat{g}_{2}$ are nonconstant and, for each $i$ $(i=1, 2)$, 
$g_{i}$ and $\hat{g}_{i}\circ \Psi$ share $p_{i} >4$ distinct values. 
If $g_{1}\not\equiv \hat{g}_{1}\circ \Psi$ and $g_{2}\not\equiv \hat{g}_{2}\circ \Psi$, then we have
\begin{equation}\label{eq-2-2}
\dfrac{1}{p_{1}-4}+\dfrac{1}{p_{2}-4}\geq 1. 
\end{equation}
In particular, if $p_{1}\geq 7$ and $p_{2}\geq 7$, then either $g_{1}\equiv \hat{g}_{1}\circ \Psi$ or 
$g_{2}\equiv \hat{g}_{2}\circ \Psi$, or both hold. 
\item[(i\hspace{-.1em}i)] Assume that $g_{1}, \hat{g}_{1}$ are nonconstant, 
and $g_{1}$ and $\hat{g}_{1}\circ \Psi$ share $p$ distinct values. 
If $g_{1}\not\equiv \hat{g}_{1}\circ \Psi$ and $g_{2}\equiv \hat{g}_{2}\circ \Psi$ is constant, then we have $p\leq 5$. 
In particular, if $p\geq 6$, then $G\equiv \widehat{G}\circ \Psi$.
\end{enumerate}
\end{theorem}

The paper is organized as follows: In Section \ref{Sec-Main}, to reveal the geometric interpretation of Theorem \ref{thm-cor}, 
we give a unicity theorem for the holomorphic map $G=(g_{1}, \ldots, g_{n})$ into 
$(\RC)^{n}:=\underbrace{\RC\times \cdots \times \RC}_{n}$ on open Riemann surfaces with the conformal metric 
$ds^{2}=\prod_{i=1}^{n}(1+|g_{i}|^{2})^{m_{i}}|\omega|^{2}$, where $\omega$ is a holomorphic $1$-form on $\Sigma$ 
and each $m_{i}$ $(i=1, \cdots, n)$ is a positive integer (Theorem \ref{thm-main}). 
By virtue of the result, Theorem \ref{thm-cor} deeply depends on the induced metric from ${\R}^{4}$. 
Moreover we give examples (Example \ref{thm-exe}) which ensure that Theorem \ref{thm-cor} is optimal. 
The proof and some remarks of Theorem \ref{thm-cor} are given in the latter of this section. 
Section \ref{Sec-Proof} provides the proof of Theorem \ref{thm-main}. The main idea of the proof is to 
construct some flat pseudo-metric on $\Sigma$ and compare it with the Poincar\'e metric. 

Finally, the authors would like to thank Professor Yasuhiro Nakagawa for his useful comments. 


\section{Main results}\label{Sec-Main}
To elucidate the geometric interpretation of Theorem \ref{thm-cor}, we give the following theorem. 

\begin{theorem}\label{thm-main}
Let $\Sigma$ be an open Riemann surface with the conformal metric 
$$
ds^{2}=\displaystyle \prod_{i=1}^{n}(1+|g_{i}|^{2})^{m_{i}}|\omega|^{2} 
$$
and $\widehat{\Sigma}$ another open Riemann surface with the conformal metric 
$$
d\hat{s}^{2}= \displaystyle \prod_{i=1}^{n}(1+|\hat{g}_{i}|^{2})^{m_{i}}|\hat{\omega}|^{2}, 
$$
where $\omega$ and $\hat{\omega}$ are holomorphic $1$-forms, $G$ and $\widehat{G}$ are holomorphic maps 
into $(\RC)^{n}:=\underbrace{\RC\times \cdots \times \RC}_{n}$ on $\Sigma$ and $\widehat{\Sigma}$ respectively, 
and each $m_{i}$ $(i=1, \cdots, n)$ is a positive integer. 
We assume that there exists a conformal diffeomorphism $\Psi \colon \Sigma \to 
\widehat{\Sigma}$, and $g_{i_{1}}, \ldots, g_{i_{k}}$ and $\hat{g}_{i_{1}}, \ldots, \hat{g}_{i_{k}}$ $(1\leq i_{1}< \cdots < i_{k}\leq n)$ are 
nonconstant and the others are constant. For each $i_{l}$ $(l=1, \cdots, k)$, we suppose that 
$g_{i_{l}}$ and $\hat{g}_{i_{l}}\circ \Psi$ share $q_{i_{l}}> 4$ distinct values and $g_{i_{l}} \not\equiv \hat{g}_{i_{l}}\circ \Psi$.  
If either $ds^{2}$ or $d\hat{s}^{2}$ is complete, then we have 
\begin{equation}\label{eq-2-1}
\displaystyle \sum_{l=1}^{k} \dfrac{m_{i_{l}}}{q_{i_{l}}-4}\geq 1. 
\end{equation}
\end{theorem}

We remark that Theorem \ref{thm-main} also holds for the case where at least one of $m_{1}, \ldots, m_{n}$ is positive 
and the others are zeros. For example, we assume that $g:= g_{i_{1}}$ and $\hat{g}:= \hat{g}_{i_{1}}$ are nonconstant and the others 
are constant. If $m:= m_{i_{1}}$ is a positive integer and the others are zeros, then the inequality (\ref{eq-2-1})  coincides with 
$$
\dfrac{m}{q-4}\geq 1 \, \Longleftrightarrow  \, q \leq m+4,  
$$
where $q:= q_{i_{1}}$. The result corresponds with \cite[Theorem 2.9]{Ka2015}.  

Theorem \ref{thm-main} is optimal because there exist the following examples. 

\begin{example}\label{thm-exe}
For positive integers $m_{1}, \ldots, m_{n}$ whose the sum $M:=\sum_{l=1}^{k}m_{i_{l}}$ of 
the subset $\{i_{1}, \ldots, i_{k}\}$ in $\{ 1, \cdots, n  \}$ is even, we take $M/2$ distinct points 
${\alpha}_{1}, \ldots, {\alpha}_{M/2}$ in $\C\backslash \{0, \pm 1\}$. Let $\Sigma$ be either the complex 
plane punctured at $M+1$ distinct points $0, {\alpha}_{1}, \ldots, {\alpha}_{M/2}, 1/{\alpha}_{1}, \ldots, 1/{\alpha}_{M/2}$ or 
the universal covering of the punctured plane. We set that 
$$
\omega =\dfrac{dz}{z\prod_{i=1}^{M/2}(z-{\alpha}_{i})({\alpha}_{i}z-1)}
$$
and, the map $G=(g_{1}, \cdots , g_{n})$ is given by 
$$
g_{i_{1}}= \cdots =g_{i_{k}}= z \quad (1\leq i_{1}< \cdots < i_{k}\leq n)
$$
and the others are constant. In a similar manner, we set 
$$
\hat{\omega}(=\omega) =\dfrac{dz}{z\prod_{i=1}^{M/2}(z-{\alpha}_{i})({\alpha}_{i}z-1)}
$$
and, the map $\widehat{G}=(\hat{g}_{1}, \cdots , \hat{g}_{n})$ is given by 
$$
\hat{g}_{i_{1}}= \cdots =\hat{g}_{i_{k}}= 1/z \quad (1\leq i_{1}< \cdots < i_{k}\leq n) 
$$
and the others are constant. We can easily show that the identity map $\Psi\colon \Sigma \to \Sigma$ is a conformal 
diffeomorphism and the metric $ds^{2}=\prod_{i=1}^{n}(1+|g_{i}|^{2})^{m_{i}}|\omega|^{2}$ is complete. 
Then for each $i_{l}$, the maps $g_{i_{l}}$ and $\hat{g}_{i_{l}}$ $(l=1, \cdots, k)$ share the $M+4$ distinct values 
$0, \infty, 1, -1, {\alpha}_{1}, \ldots, {\alpha}_{M/2}, 1/{\alpha}_{1}, \ldots, 1/{\alpha}_{M/2}$ and 
$g_{i_{l}} \not\equiv \hat{g}_{i_{l}}\circ \Psi$. These show that Theorem \ref{thm-main} is optimal. 
\end{example}

We will apply Theorem \ref{thm-main} to the Gauss maps of complete minimal surfaces in ${\R}^{4}$. 
We first recall some basic facts of minimal surfaces in ${\R}^{4}$. 
For more details, we refer the reader to \cite{Ch1965, HO1980, HO1985, Os1964}. 
Let $X=(x^{1}, x^{2}, x^{3}, x^{4})\colon \Sigma \to {\R}^{4}$ be an oriented minimal 
surface in ${\R}^4$. By associating a local complex coordinate $z=u+\sqrt{-1}v$ with each positive isothermal coordinate system 
$(u, v)$, $\Sigma$ is considered as a Riemann surface whose conformal metric is the induced metric $ds^{2}$ from ${\R}^{4}$. 
Then 
\begin{equation}\label{equ-appl-min-1}
\triangle_{ds^{2}} X=0
\end{equation}
holds, that is, each coordinate function $x^{i}$ is harmonic. With respect to the local coordinate $z$ of the surface, 
(\ref{equ-appl-min-1}) is given by 
$$
\bar{\partial} \partial X =0, 
$$
where $\partial =(\partial /\partial u - \sqrt{-1}\partial /\partial v)/2$, $\bar{\partial} 
=(\partial /\partial u + \sqrt{-1}\partial /\partial v)/2$. Hence each ${\phi}_{i}:= \partial x^{i} dz$ ($i=1, 2, 3, 4$) is a 
holomorphic $1$-form on $\Sigma$. If we set that 
$$
\omega = {\phi}_{1} -\sqrt{-1} {\phi}_{2}, \qquad g_{1}=\dfrac{{\phi}_{3}+\sqrt{-1}{\phi}_{4}}{{\phi}_{1} -\sqrt{-1} {\phi}_{2}}, 
\qquad g_{2}=\dfrac{-{\phi}_{3}+\sqrt{-1}{\phi}_{4}}{{\phi}_{1} -\sqrt{-1} {\phi}_{2}}, 
$$
then $\omega$ is a holomorphic $1$-form, and $g_{1}$ and $g_{2}$ are meromorphic functions on $\Sigma$. 
Moreover the holomorphic map $G:=(g_{1}, g_{2})\colon \Sigma \to \RC \times \RC$ coincides with the Gauss map of $X(\Sigma)$. 
We remark that the Gauss map of $X(\Sigma)$ in ${\R}^{4}$ is the map from each point of $\Sigma$ to its oriented tangent plane, 
the set of all oriented (tangent) planes in ${\R}^{4}$ is naturally identified with the quadric 
$$
\mathbf{Q}^{2}(\C) =\{[w^{1}: w^{2}: w^{3}: w^{4}] \in \mathbf{P}^{3}(\C) \, ;\, (w^{1})^{2}+\cdots +(w^{4})^{2} = 0\}
$$
in $\mathbf{P}^{3}(\C)$, and $\mathbf{Q}^{2}(\C)$ is biholomorphic to the product of the Riemann spheres $\RC \times \RC$. 
Furthermore the induced metric from ${\R}^{4}$ is given by 
\begin{equation}\label{equ-appl-min-2}
ds^{2}= (1+|g_{1}|^{2})(1+|g_{2}|^{2})|\omega|^{2}. 
\end{equation}

Applying Theorem \ref{thm-main} to the induced metric, we obtain Theorem \ref{thm-cor}.  

\begin{proof}[{\it Proof of Theorem \ref{thm-cor}}] 
We first show the case (i). Since $m_{1}=m_{2}=1$ from (\ref{equ-appl-min-2}), 
we can prove the inequality (\ref{eq-2-2}) by Theorem \ref{thm-main}. Next we show the case (i\hspace{-.1em}i). 
By Theorem \ref{thm-main}, we obtain 
$$
\dfrac{1}{p-4}\geq 1. 
$$
Thus we have $p\leq 4+1=5$. 
\end{proof}

\begin{remark}\label{rmk-R-m}
Fujimoto \cite{Fu1993-3} obtained the result of the unicity theorem for the Gauss maps $G\colon \Sigma \to {\mathbf{P}}^{m-1}(\C)$ 
of complete minimal surfaces in ${\R}^{m}$ $(m\geq 3)$. However the result does not contain Theorem \ref{thm-cor} because 
corresponding hyperplanes in ${\mathbf{P}}^{3}(\C)$ are not necessary located in general position (For more details, 
see \cite[Page 353]{MO1990}). 

\end{remark}

\section{Proof of Theorem \ref{thm-main}}\label{Sec-Proof}


We first recall the notion of chordal distance between two distinct points in $\RC$. 
For two distinct points $\alpha$, $\beta\in \RC$, we set 
$$
|\alpha, \beta|:= \dfrac{|\alpha -\beta|}{\sqrt{1+|\alpha|^{2}}\sqrt{1+|\beta|^{2}}}
$$
if $\alpha \not= \infty$ and $\beta \not= \infty$, and $|\alpha, \infty|=|\infty, \alpha| := 1/\sqrt{1+|\alpha|^{2}}$. 
We note that, if we take $v_{1}$, $v_{2}\in {\Si}^{2}$ with $\alpha =\varpi (v_{1})$ and $\beta = \varpi (v_{2})$, 
we have that $|\alpha, \beta|$ is a half of the chordal distance between $v_{1}$ and $v_{2}$, 
where $\varpi$ denotes the stereographic projection of the $2$-sphere ${\Si}^{2}$ onto $\RC$. 


We next review the following three lemmas used in the proof of Theorem \ref{thm-main}. 
\begin{lemma}[{\cite[Proposition 2.1]{Fu1993-2}}]\label{main-lem1}
Let $g_{i_{l}}$ and $\hat{g}_{i_{l}}$ be mutually distinct nonconstant meromorphic functions on a Riemann surface $\Sigma$. 
Let $q_{i_{l}}$ be a positive integer and ${\alpha}^{l}_{1}, \ldots, {\alpha}^{l}_{q_{i_{l}}}\in \RC$ be distinct. Suppose that 
$q_{i_{l}}>4$ and $g_{i_{l}}^{-1}({\alpha}^{l}_{j})={\hat{g}_{i_{l}}}^{-1}({\alpha}^{l}_{j})$ $(1\leq j\leq q_{i_{l}})$. 
For $a^{l}_{0}>0$ and $\varepsilon$ with $q_{i_{l}}-4> q_{i_{l}}\varepsilon >0$, we set that 
$$
\xi_{i_{l}} :=\displaystyle \biggl{(}\prod_{j=1}^{q_{i_{l}}} |g_{i_{l}}, {\alpha}^{l}_{j}| 
\log{\biggl{(}\dfrac{a^{l}_{0}}{|g_{i_{l}}, {\alpha}^{l}_{j}|^{2}} \biggr{)}}\biggr{)}^{-1+\varepsilon}, 
\quad \hat{\xi}_{i_{l}} := \displaystyle \biggl{(}\prod_{j=1}^{q_{i_{l}}} |\hat{g}_{i_{l}}, {\alpha}^{l}_{j}| \log{\biggl{(}
\dfrac{a^{l}_{0}}{|\hat{g}_{i_{l}}, {\alpha}^{l}_{j}|^{2}} \biggr{)}}\biggr{)}^{-1+\varepsilon}, 
$$
and define that 
\begin{equation}\label{equ-proof-2-1}
d\tau^{2}_{i_{l}}:=\biggl{(}|g_{i_{l}}, \hat{g}_{i_{l}}|^{2}\xi_{i_{l}}\, \hat{\xi}_{i_{l}}
\dfrac{|g_{i_{l}}'|}{1+|g_{i_{l}}|^{2}}\dfrac{|\hat{g}'_{i_{l}}|}{1+|\hat{g}_{i_{l}}|^{2}}\biggr{)}|dz|^{2}
\end{equation}
outside the set $E:=\bigcup_{j=1}^{q}g_{i_{l}}^{-1}({\alpha}^{l}_{j})$ and $d\tau^{2}_{i_{l}}=0$ on $E$. 
Then for a suitably chosen $a_{0}$, $d\tau^{2}_{i_{l}}$ is continuous on $\Sigma$ and has strictly negative curvature on 
the set $\{d\tau^{2}_{i_{l}}\not= 0\}$. 
\end{lemma}

\begin{lemma}[{\cite[Corollary 2.4]{Fu1993-2}}]\label{main-lem2}
Let $g_{i_{l}}$ and $\hat{g}_{i_{l}}$ be meromorphic functions on ${\triangle}_{R}$ 
satisfying the same assumption as in Lemma \ref{main-lem1}. 
Then for the metric $d\tau^{2}$ defined by (\ref{equ-proof-2-1}), there exists a constant $C>0$ such that 
$$
d\tau^{2}_{i_{l}}\leq C\dfrac{R^{2}}{(R^{2}-|z|^{2})^{2}}|dz|^{2}. 
$$
\end{lemma}

\begin{lemma}{\cite[Lemma 1.6.7]{Fu1993}}\label{Lem-main2} 
Let $d{\sigma}^{2}$ be a conformal flat-metric on an open Riemann surface $\Sigma$. 
Then, for each point $p\in \Sigma$, there exists a local diffeomorphism $\Phi$ of a 
disk ${\Delta}_{R}=\{z\in \C ; |z|< R \}$ $(0<R\leq +\infty)$ onto an open 
neighborhood of $p$ with $\Phi (0)=p$ such that $\Phi$ is an isometry, that is, 
the pull-back ${\Phi}^{\ast}(d{\sigma}^{2})$ is equal to the standard Euclidean metric $ds^{2}_{E}$ on ${\Delta}_{R}$ 
and that, for a specific point $a_{0}$ with $|a_{0}|=1$, the ${\Phi}$-image ${\Gamma}_{a_{0}}$ of 
the curve $L_{a_{0}}=\{w:= a_{0}s ; 0 < s < R\}$ 
is divergent in $\Sigma$. 
\end{lemma}

\begin{proof}[{\it Proof of Theorem \ref{thm-main}}] 
Since the given map $\Psi$ gives a biholomorphic isomorphism between $\Sigma$ and 
$\widehat{\Sigma}$, we denote the function $\hat{g}_{i_{l}}\circ \Psi$ by 
$\hat{g}_{i_{l}}$ $(l =1, \cdots, k)$ for brevity. For each $i_{l}$, we assume that 
$g_{i_{l}}$ and $\hat{g}_{i_{l}}$ share the $q_{i_{l}}$ distinct values ${\alpha}^{l}_{1}, \ldots, {\alpha}^{l}_{q_{i_{l}}}$. 
After suitable M\"obius transformations for $g_{i_{l}}$ and $\hat{g}_{i_{l}}$, 
we may assume that ${\alpha}_{q_{i_{1}}}^{1}=\cdots ={\alpha}_{q_{i_{k}}}^{k}=\infty$. 
Moreover we assume that either $ds^{2}$ or $d\hat{s}^{2}$, say $ds^{2}$ is complete and 
$g_{i_{l}}\not\equiv \hat{g}_{i_{l}}\circ \Psi$ for each $l$ $(1\leq l\leq k)$. 
Thus, for each local complex coordinate $z$ defined on a simply connected open domain 
$U$, we can find a nonzero holomorphic function $h_{z}$ such that 
\begin{equation}\label{eq-proof-2-2}
\displaystyle ds^{2}=|h_{z}|^{2}\prod_{i=1}^{n}(1+|g_{i}|^{2})^{m_{i}/2}(1+|\hat{g}_{i}|^{2})^{m_{i}/2}|dz|^{2}. 
\end{equation} 
Suppose that each $q_{i_{l}}> 4$ and 
\begin{equation}\label{eq-proof-2-3}
\displaystyle \sum_{l=1}^{k} \dfrac{m_{i_{l}}}{q_{i_{l}}-4}< 1. 
\end{equation}
Then, by (\ref{eq-proof-2-3}), we may suppose that $q_{i_{l}}> m_{i_{l}}+4$ for each $i_{l}$ $(l=1, \cdots, k)$. 
Taking some positive number $\eta_0$ with
\begin{equation}\label{eq:1}
0< \eta_0 < \dfrac{q_{i_{l}}-4-m_{i_{l}}}{q_{i_{l}}}
\end{equation}
for each $i_{l}$ $(l=1,\cdots, k)$ and
\begin{equation}\label{eq:2}
\displaystyle {\Lambda}_{0}:= \sum_{l=1}^{k} \dfrac{m_{i_{l}}}{q_{i_{l}}-4-q_{i_{l}}\eta_0}=1. 
\end{equation}
For a positive number $\eta$ with $\eta < {\eta}_{0}$, we set 
$$
{\lambda}_{i_{l}}:= \dfrac{m_{i_{l}}}{q_{i_{l}}-4-q_{i_{l}}\eta}\quad (l=1, \cdots, k). 
$$
By (\ref{eq:2}) we get
\begin{equation}\label{eq-proof-2-5}
\Lambda :=\displaystyle \sum_{l=1}^{k} {\lambda}_{i_{l}} = \sum_{l=1}^{k} 
\dfrac{m_{i_{l}}}{q_{i_{l}}-4-q_{i_{l}}\eta}<\sum_{l=1}^{k} \dfrac{m_{i_{l}}}{q_{i_{l}}-4-q_{i_{l}}\eta_0}={\Lambda}_{0}= 1.
\end{equation}
Now we can choose a positive number $\eta (< \eta_0)$ sufficiently near $\eta_0$ satisfying 
\begin{equation}\label{eq:3}
{\Lambda}_{0}-{\Lambda} <\min_{1\leq t\leq k} \biggl{\{}\dfrac{m_{i_{t}}}{q_{i_{t}}-4-q_{i_{t}}\eta},\, \dfrac{m_{i_{t}}\eta}{q_{i_{t}}-4-q_{i_{t}}\eta}\biggr{\}}.
\end{equation}
Using (\ref{eq-proof-2-5}) and (\ref{eq:3}), we have
\begin{equation}\label{eq-proof-2-6}
\dfrac{{\lambda}_{i_{l}}}{1-\Lambda}> 1 \quad \text{ and } \quad \dfrac{\eta{\lambda}_{i_{l}}}{1-\Lambda}> 1 \quad (l=1, \cdots, k).  
\end{equation}

Now we define a new metric 
\begin{equation}\label{eq-proof-2-7}
\displaystyle d{\sigma}^{2}= |h_{z}|^{\frac{4}{1-\Lambda}}\prod_{l=1}^{k}
\Biggl{(}\dfrac{\prod_{j=1}^{q_{i_{l}}-1}(|g_{i_{l}}-{\alpha}^{l}_{j}| |\hat{g}_{i_{l}}-{\alpha}^{l}_{j}|)^{1-\eta}}{|g_{i_{l}}-\hat{g}_{i_{l}}|^{2}|g'_{i_{l}}||\hat{g}'_{i_{l}}|
\prod_{j=1}^{q_{i_{l }-1}} (1+|\alpha_{j}^{l}|^{2})^{1-\eta}} \Biggr{)}^{\frac{2{\lambda}_{i_{l}}}{1-\Lambda}}|dz|^{2} 
\end{equation}
on the set ${\Sigma}'=\Sigma \backslash E$, where 
$$
E=\{p\in \Sigma \, ; \, g'_{i_{l}}(p)= 0,\,\,  \hat{g}'_{i_{l}}(p)= 0 \,\, \text{or} \,\, g_{i_{l}}(p) (=\hat{g}_{i_{l}}(p) )= {\alpha}^{l}_{j} \,\, \text{for some}\,\, l\}. 
$$
On the other hand, setting $\varepsilon :=\eta /2$, we can define another pseudo-metric $d{\tau}^{2}_{i_{l}}$ on $\Sigma$ given by (\ref{equ-proof-2-1}) 
for each $l$, which has strictly negative curvature on ${\Sigma}'$. Take a point $p\in {\Sigma}'$. Since the metric $d{\sigma}^{2}$ is flat on ${\Sigma}'$, 
by Lemma \ref{Lem-main2}, there exists an isometry $\Phi$ satisfying $\Phi (0)= p$ from a disk $\triangle_{R}=\{z\in \C \, ; \, |z|< R\}$ $(0< R\leq +\infty)$ with 
the standard metric $ds^{2}_{E}$ on an open neighborhood of $p$ in ${\Sigma}'$ with the metric $d{\sigma}^{2}$. We denote the functions $g_{i_{l}}\circ \Phi$ and 
 $\hat{g}_{i_{l}}\circ \Phi (=\hat{g}_{i_{l}}\circ \Psi \circ \Phi)$ by $g_{i_{l}}$ and $\hat{g}_{i_{l}}$ respectively $(l=1, \cdots, k)$ in the following. 
Moreover, for each $i_{l}$, the pseudo-metric $d{\sigma}^{2}_{i_{l}}$ on $\triangle_{R}$ has strictly negative curvature. Since there exists no metric with 
strictly negative curvature on $\C$ (see \cite[Corollary 4.2.4]{Fu1993}),  we get that the radius $R$ is finite. Furthermore, by Lemma \ref{Lem-main2}, 
we can choose a point $a_{0}$ with $|a_{0}|=1$ such that, for the line segment $L_{a_{0}}:=\{w := a_{0}s \, ;\, 0<s<R \}$, 
the $\Phi$-image $\Gamma_{a_{0}}$ tends 
to the boundary of ${\Sigma}'$ as $s$ tends to $R$. Then ${\Gamma}_{a_{0}}$ is divergent in $\Sigma$. 
Indeed, if not, then ${\Gamma}_{a_{0}}$ must tend to a point $p_{0}\in E$. 
Then we consider the following two possible cases: 

Case 1:\, $g_{i_{l}}(p_{0}) (=\hat{g}_{i_{l}}(p_{0}) )= {\alpha}^{l}_{j} \,\, \text{for some}\,\, l.$\\
Since $g'_{i_{l}}(p_0) = (g_{i_{l}}-\alpha^{l}_{j})'(p_0)$ and 
$\hat{g}'_{i_{l}}(p_0) = (\hat{g}_{i_{l}}-\alpha^{l}_{j})'(p_0)$, the function 
$$
 |h_{z}|^{\frac{2}{1-\Lambda}}\prod_{l=1}^{k}
\Biggl{(}\dfrac{\prod_{j=1}^{q_{i_{l}}-1}(|g_{i_{l}}-{\alpha}^{l}_{j}| |\hat{g}_{i_{l}}-{\alpha}^{l}_{j}|)^{1-\eta}}{|g_{i_{l}}-\hat{g}_{i_{l}}|^{2}|g'_{i_{l}}||\hat{g}'_{i_{l}}|
\prod_{j=1}^{q_{i_{l }-1}} (1+|\alpha_{j}^{l}|^{2})^{1-\eta}} \Biggr{)}^{\frac{{\lambda}_{i_{l}}}{1-\Lambda}}
$$ 
has a pole of order at least $2\eta{\lambda}_{i_{l}}/(1-\Lambda)$ at $p_{0}$.  
Taking a local complex coordinate $\zeta$ in a neighborhood of $p_{0}$ with $\zeta (p_{0}) =0$, 
we can write the metric $d{\sigma}^{2}$ as 
$$
d{\sigma}^{2}= |\zeta|^{-4\eta{\lambda}_{i_{l}}/(1-\Lambda)} w|d\zeta|^{2} 
$$
with some positive function $w$. Since $\eta{\lambda}_{i_{l}}/(1-\Lambda) >1$, we have  
$$
R = \int_{{\Gamma}_{a_{0}}} d\sigma > C_{1}\int_{\Gamma_{a_{0}}} |d\zeta|/ |\zeta|^{2\eta{\lambda}_{i_{l}}/(1-\Lambda)} = +\infty. 
$$
This contradicts that $R$ is finite. 

Case 2:\, $g'_{i_{l}}(p_{0})\hat{g}'_{i_{l}}(p_{0})= 0$ for some $i_{l}.$ \\
Without loss of generality, we may assume that $g'_{i_{l}}(p_{0})= 0$ for some $i_{l}$. 
Taking a local complex coordinate $\zeta :=g'_{i_{l}}$ in a neighborhood 
of $p_{0}$ with $\zeta (p_{0}) =0$, we can write the metric $d{\sigma}^{2}$ as 
$$
d{\sigma}^{2}= |\zeta|^{-2{\lambda}_{i_{l}}/(1-\Lambda)} w|d\zeta|^{2} 
$$
with some positive function $w$. Since ${\lambda}_{i_{l}}/(1-\Lambda) >1$, we have  
$$
R = \int_{{\Gamma}_{a_{0}}} d\sigma > C_{2}\int_{\Gamma_{a_{0}}} |d\zeta|/ |\zeta|^{{\lambda}_{i_{l}}/(1-\Lambda)} = +\infty. 
$$
This also contradicts that $R$ is finite. 

Since ${\Phi}^{\ast} d{\sigma}^{2} =|dz|^{2}$, we get by (\ref{eq-proof-2-7}) that 
$$
|h_{z}|^{2}= \displaystyle \prod_{l=1}^{k} \Biggl{(}\dfrac{|g_{i_{l}}-\hat{g}_{i_{l}}|^{2} |g'_{i_{l}}| |\hat{g}'_{i_{l}}|\prod_{j=1}^{q_{i_{l}}-1}
(1+|{\alpha}^{l}_{j}|^{2})^{1-\eta}}{\prod_{j=1}^{q_{i_{l}}-1}(|g_{i_{l}}-{\alpha}^{l}_{j}| |\hat{g}_{i_{l}}-{\alpha}^{l}_{j}|)^{1-\eta}} \Biggr{)}^{{\lambda}_{i_{l}}}. 
$$
By (\ref{eq-proof-2-2}), we have 
\begin{eqnarray*}
{\Phi}^{\ast} ds &=& |h_{z}|^{2} \prod_{i=1}^{n} (1+|g_{i}|^{2})^{m_{i}/2}(1+|\hat{g}_{i}|^{2})^{m_{i}/2}|dz|^{2} \\
                     &\leq& C_{3} \prod_{l=1}^{k} \Biggl{(}\dfrac{|g_{i_{l}}-\hat{g}_{i_{l}}|^{2} |g'_{i_{l}}| |\hat{g}'_{i_{l}}| 
(1+|g_{i_{l}}|^{2})^{m_{i_{l}}/2{\lambda}_{i_{l}}}  (1+|\hat{g}_{i_{l}}|^{2})^{m_{i_{l}}/2{\lambda}_{i_{l}}} \prod_{j=1}^{q_{i_{l}}-1}
(1+|{\alpha}^{l}_{j}|^{2})^{1-\eta}}{\prod_{j=1}^{q_{i_{l}}-1}(|g_{i_{l}}-{\alpha}^{l}_{j}| |\hat{g}_{i_{l}}-{\alpha}^{l}_{j}|)^{1-\eta}} \Biggr{)}^{{\lambda}_{i_{l}}}|dz|^{2} \\
                     &=& C_{3} \prod_{l=1}^{k} \Biggl{(}{\mu}_{i_{l}}^{2} \prod_{j=1}^{q_{i_{l}}} \biggl{(} |g_{i_{l}}, {\alpha}^{l}_{j}| 
|\hat{g}_{i_{l}}, {\alpha}^{l}_{j}| \biggr{)}^{\varepsilon} \biggl{(}\log \biggl{(} \dfrac{a^{l}_{0}}{|g_{i_{l}}, {\alpha}^{l}_{j}|} \biggr{)} 
\log \biggl{(} \dfrac{a^{l}_{0}}{|\hat{g}_{i_{l}}, {\alpha}^{l}_{j}|} \biggr{)} \biggr{)}^{1-\varepsilon} \Biggr{)}^{{\lambda}_{i_{l}}} |dz|^{2}, 
\end{eqnarray*}
where ${\mu}_{i_{l}}$ is the function with $d{\tau}^{2}_{i_{l}} ={\mu}^{2}_{i_{l}} |dz|^{2}$. Since the function 
$x^{\varepsilon} \log^{1-\varepsilon}(a^{l}_{0}/ x)$ $(0< x\leq 1)$ is bounded, we obtain that 
$$
ds^{2}\leq C_{4}\prod_{l=1}^{k} \Biggl{(}\dfrac{|g_{i_{l}}, \hat{g}_{i_{l}}|^{2} |g'_{i_{l}}| |\hat{g}'_{i_{l}}| {\xi}_{i_{l}} \hat{\xi}_{i_{l}} }{(1+|g_{i_{l}}|^{2})(1+|\hat{g}_{i_{l}}|^{2})  } \Biggr{)}^{{\lambda}_{i_{l}}} |dz|^{2}
$$
for some $C_{4}$. By Lemma \ref{main-lem2}, we have that 
$$
ds^{2}\leq C_{5} \prod_{l=1}^{k} \Biggl{(}\dfrac{R}{R^{2}-|z|^{2}} \Biggr{)}^{{\lambda}_{i_{l}}} |dz|^{2} = C_{5} \Biggl{(}\dfrac{R^{2}}{R^{2}-|z|^{2}} \Biggr{)}^{\Lambda}|dz|^{2}
$$
for some constant $C_{5}$. Thus we obtain that 
$$
\int_{{\Gamma}_{a_{0}}} ds\leq (C_{5})^{1/2}\int_{L_{a_{0}} }\Biggl{(}\dfrac{R^{2}}{R^{2}-|z|^{2}} \Biggr{)}^{\Lambda /2}|dz|< 
C_{6} \dfrac{R^{(2-\Lambda)/2}}{1-(\Lambda /2)}< +\infty
$$
because $0< \Lambda < 1$. However it contradicts the assumption that the metric $ds^{2}$ is complete. 
\end{proof}


\end{document}